\documentclass{amsart}
\usepackage{amssymb,enumerate}

\usepackage[all,2cell,ps]{xy}
\usepackage{hyperref}

\newtheorem{theorem}[subsection]{Theorem}

\newtheorem{proposition}[subsection]{Proposition}
\newtheorem{lemma}[subsection]{Lemma}
\newtheorem{corollary}[subsection]{Corollary}

\theoremstyle{definition}

\newtheorem{example}[subsection]{Example}
\newtheorem{conjecture}[subsection]{Conjecture}

\theoremstyle{remark}
\newtheorem{remark}[subsection]{Remark}

\numberwithin{equation}{subsection}

\newcommand{\lra}{\longrightarrow}
\newcommand{\xra}{\xrightarrow}

\newcommand{\fm}{\mathfrak{m}}
\newcommand{\fn}{\mathfrak{n}}
\newcommand{\fp}{\mathfrak{p}}
\newcommand{\fq}{\mathfrak{q}}

\newcommand{\calC}{\mathcal{C}}
\newcommand{\calH}{\mathcal{H}}

\newcommand{\CC}{\mathbb{C}}
\newcommand{\ZZ}{\mathbb{Z}}

\newcommand{\con}[2]{\operatorname{C}(#1,#2)}

\newcommand{\depth}{\operatorname{depth}}
\newcommand{\edim}{\operatorname{embdim}}
\newcommand{\End}{\operatorname{End}}
\newcommand{\Ext}{\operatorname{Ext}}
\newcommand{\height}{\operatorname{height}}
\newcommand{\Hom}{\operatorname{Hom}}

\newcommand{\im}{\operatorname{im}}

\newcommand{\Tor}{\operatorname{Tor}}
\newcommand{\rank}{\operatorname{rank}}
\newcommand{\Tr}{\operatorname{tr}}

\newcommand{\ul}{\underline}
\newcommand{\ov}{\overline}
\newcommand{\dual}[1]{#1^{\vee}}

\begin{document}

\title[Rigid ideals ]
{Rigid ideals in Gorenstein rings \\of dimension one}
\author[Huneke, Iyengar,  Wiegand]
{Craig Huneke, Srikanth B. Iyengar, and Roger Wiegand}

\address{Craig Huneke\\
Department of Mathematics \\
University of Virginia\\
Charlottsville, VA 22904, U.S.A.}
\email{huneke@virginia.edu }

\address{Srikanth B. Iyengar \\
Department of Mathematics \\
University of Utah\\
Salt Lake City, UT 84112, U.S.A.}
\email{iyengar@math.utah.edu}

\address{Roger Wiegand\\
Department of Mathematics \\
University of Nebraska \\
Lincoln, NE 68588, U.S.A.}
\email{rwiegand1@math.unl.edu}

\subjclass[2010]{13D07, 13C14, 13C99}

\keywords{complete intersection ring, Gorenstein ring,  rigid module, tensor product, torsion}

\date{\today}

\begin{abstract}
We investigate the existence of  ideals $I$ in a one-dimensional Gorenstein local ring $R$ satisfying $\mathrm{Ext}^{1}_{R}(I,I)=0$.
\end{abstract}

\maketitle{}

\section{Introduction}

In their study of torsion in tensor products in \cite{HW1}, the first and third authors of this paper
proposed the following  conjecture:

\begin{conjecture} 
\label{conj:HW} Let 
$R$ be a   Gorenstein local domain of dimension one and $M$ a finitely generated 
$R$-module such that $M\otimes_RM^*$ is torsionfree. Then $M$ is free.  
\end{conjecture}

Here and throughout this paper we set $M^* = \Hom_R(M,R)$. An $R$-module is \emph{torsionfree} if no regular element of $R$ kills a nonzero element of the module. 

An $R$-module  $M$  is \emph {rigid} provided $\Ext^1_R(M,M) = 0$, that is to say, every self-extension of $M$ splits.  Over a one-dimensional Gorenstein local ring, a finitely generated  torsionfree $R$-module $M$ is rigid if and only if $M\otimes_RM^*$ is torsionfree; see Proposition \ref{prp:tf=rigid}.  Conjecture \ref{conj:HW}  reduces to the case where  $M$ is torsionfree (see \ref{ch:red-to-tf}), and hence is equivalent to the following:

\begin{conjecture} 
\label{conj:HWrigid} Over a one-dimensional local Gorenstein domain, every finitely generated 
rigid torsionfree $R$-module is free.  
\end{conjecture}

Conjecture \ref{conj:HW} is proved for abstract hypersurfaces in \cite{HW1}.  More generally in that paper it is shown that if $R$ is an abstract hypersurface of dimension one, and $M$ and $N$ are finitely generated $R$-modules, at least one of which has  rank, and if $M\otimes_RN$ is  torsionfree, then either $M$ or $N$ is free.  Celikbas \cite[Proposition 4.17]{Cel} has shown that Conjecture \ref{conj:HW}
holds if $R$ is a complete intersection and the module $M$ is torsion-free and has bounded Betti numbers.  

Both conjectures fail without the hypothesis that $R$ be one-dimensional; see Examples \ref{eg:dim3} and \ref{eg:double-sharp}.  A plausible way to extend Conjecture \ref{conj:HW} to higher dimensions is to replace ``torsionfree'' with ``reflexive'' (see Conjecture \ref{conj:HWdim>1}), but in fact that version of the conjecture reduces to the one-dimensional case; see \ref{ch:red-dim1}. 

In this paper we concentrate chiefly on the case of ideals, though some of our results are true more generally.  We usually insist that the ideal contain a regular element, since otherwise there are easy counterexamples; see Example \ref{eg:2lines}.

Given a local ring $R$, we write $\edim R$ for embedding dimension, and $e(R)$ for the multiplicity, of $R$.  The minimal number of generators required for an $R$-module $M$ is denoted $\nu_{R}M$. Our main results  are summarized below:

\begin{theorem}
\label{thm:main}
Let $R$ be a local ring with $\dim R=1$. Let $I$ be a reflexive ideal containing a regular element, and
assume that $I\otimes_{R}I^{*}$ is torsionfree. The ideal $I$ is then  principal under each of the following conditions:
\begin{enumerate}[\quad\rm(1)]
\item
$\nu_RI\le 2$ and $\nu_RI^*\le 2$;
\item
$e(R)\le 8$;
\item
$R$ is Gorenstein of minimal multiplicity;
\item
$R$ is a complete intersection and $e(R)\le 10$;
\item
$R$ is a complete intersection and the pair $(R,I)$ can be smoothed; 
\item
$R$ is a complete intersection with $\edim R\le 3$ and $R/I$ is Gorenstein.
\end{enumerate}
\end{theorem}

Parts (1) and (2) are Proposition \ref{prp:2gen-refl} and Theorem \ref{thm:HW-general} respectively. Part (3) is proved as Theorem~\ref{thm:HW-gorenstein};  part (4) is contained in Theorem~\ref{thm:ci-smoothable}.  Parts (5) and (6) are found in  Theorem \ref{thm:smoothable-rigid} and Corollary \ref{cor:embdim-three}.

The paper is organized as follows. Section \ref{sec:examples} contains some examples justifying the hypotheses in the conjectures, and demonstrates that the pivotal case in these considerations occurs in dimension one.  In Section \ref{sec:small-mult} we prove a number of estimates on the number of generators of rigid ideals or modules for general one-dimensional rings, in terms of the maximum number of generators of any ideal containing the trace ideal of the module studied. These are used throughout the remainder of the paper, and in particular in Section \ref{sec:small-mult} to prove Theorem~\ref{thm:main}(2). These results extend and were inspired by work of S. Goto, R. Takahashi, N. Taniguchi, and H. Le Truong~\cite[Theorem 1.4]{GTTT} which covers the case $e(R) \le 6$.

In Section \ref{sec:gorenstein} we apply the same estimates to ideals in one-dimensional Gorenstein rings. We focus on the relationship between  Gorenstein ideals and two-generated ideals through linkage. Our main result in this section is about Gorenstein rings of minimal multiplicity.

The key result of Section \ref{sec:rigidity} is a necessary and sufficient criterion for an ideal $I$ in a one-dimensional Gorenstein ring to be rigid, in terms of the lengths of certain associated modules, especially the twisted conormal module, $I\otimes_R\omega_{R/I}$.  This result is used in Section \ref{sec:smoothable} to prove that rigid smoothable ideals are principal if the ambient ring is a complete intersection. This criterion leads to our results on rigidity for complete intersections of multiplicity at most $10$. 

\subsubsection*{Acknowledgements}
This article is based on work supported by the National Science Foundation under Grant No. 0932078000, while the authors were in residence at the  Mathematical Sciences Research Institute in Berkeley, California,  during the Fall semester of 2012. CH was partially supported by NSF grant DMS-1460638; SBI partly supported by NSF grant DMS-1700985;  RW partly supported by Simons Collaboration Grants 209213 and 426885. We thank Olgur Celikbas for helpful comments during the preparation of this work.

\section{Examples and reductions}
\label{sec:examples}
Throughout the paper, $(R,\fm,k)$ is a local ring and $Q(R)$ denotes the ring of fractions of $R$. We deal only with finitely generated modules.  An $R$-module $M$ has \emph{rank} if $M_\fp$ is free of constant rank $r$ for all associated primes $\fp$ of $R$; in this case, we write $\rank_{R}(M) = r$.  We say $M$ is Cohen-Macaulay (abbreviated to CM) if  it satisfies $\depth_{R}M \ge \dim M$, and \emph{maximal Cohen-Macaulay} (abbreviated to MCM) if $\depth_{R}M\ge \dim R$.  In both cases, equality holds if $M$ is nonzero.

In this section we present some examples and observations 
that justify the assumptions on the rings and modules in
the conjectures stated in the introduction.  The first one shows why the conjectures would fail without the assumption that $R$ is a domain.  Actually, we typically get by with the weaker assumption that the 
module in question has rank, a condition that fails in this example.

\begin{example}
\label{eg:2lines} 
Let $R = \CC[[x,y]]/(xy)$.  Put $M = R/(x) = (y)$.  Then $M \cong M^*$ and $M$ is not free, but $M\otimes_RM^*\cong M$, which is torsionfree.
\end{example}

\begin{example}
\label{eg:double-sharp} 
With $R = \CC[[x,y]]/(xy)$ and $M = R/(x)$ as in Example \ref{eg:2lines}, put  $S = \CC[[x,y,u,v]]/(xy-u^2-v^2) \cong \CC[[x,y,u,v]]/(xy-uv)$.  The stable category of MCM $S$-modules is equivalent to that of MCM  $R$-modules, by Kn\"orrer periodicity \cite{Kn}.  One can check directly that $M$ is a rigid $R$-module.  Thus $S$, a three-dimensional Gorenstein domain, has a rigid MCM module that is not free.
\end{example}

\begin{remark}
\label{rem:rigid-canon}
The Gorenstein hypothesis in Conjecture \ref{conj:HWrigid} is essential.  In fact, if $R$ is any CM ring with canonical module $\omega_R$, then $\omega_R$ is rigid, and it is free only if $R$ is Gorenstein.
\end{remark}

\begin{subsection}{Reduction to the torsionfree case}
\label{ch:red-to-tf}  
A standard argument going back at least to Auslander~\cite{Au} shows that we may as well assume that the module $M$ is torsionfree.  In detail, let $\boldsymbol\top M$ denote the torsion submodule of $M$  (the kernel of $M\mapsto Q(R) \otimes_RM$) and $\boldsymbol\bot M = M/\boldsymbol\top M$.  If $N$ is another module and $M\otimes_RN$ is torsionfree, then the map  $M\otimes_R N \to (\boldsymbol\bot M)\otimes_R(\boldsymbol\bot N)$ is an isomorphism; see \cite[Lemma 1.2]{CW}.  In particular, $ (\boldsymbol\bot M)\otimes_R(\boldsymbol\bot N)$ is torsionfree as well.  If, now, $\boldsymbol\bot M$ is free, then $M \cong \boldsymbol\bot M \oplus \boldsymbol\top M$, and $M\otimes_RN$ torsionfree implies that $\boldsymbol\top M = 0$.  Of course $M^*$, being a syzygy, is always torsionfree.  
\end{subsection}

\begin{example}
\label{eg:dim3}  
In dimensions greater than one, Conjecture \ref{conj:HW}  can fail without stronger depth conditions on $M\otimes_RM^*$:  For the ring $S=\CC[[x,y,u,v]]/(xy-uv)$ and the ideal $I = (x,u)$, the tensor product $I\otimes_RI^*$ is isomorphic to the maximal ideal $\fm$ and hence is torsionfree (but not reflexive).  See \cite[Example 4.1]{HW1}. 
\end{example}

The last example suggests higher-dimensional version of Conjecture \ref{conj:HW}:

\begin{conjecture}
\label{conj:HWdim>1} 
If $R$ is CM, $\dim R \ge 1$ and $M\otimes_RM^*$ is reflexive, then $M$ is free.  
\end{conjecture}

\begin{subsection}{Reduction to Dimension One}
\label{ch:red-dim1}
Several authors have pointed out that this conjecture reduces to dimension one;  see \cite[Propostion 8.6]{CW}.   One observes that $M$ is projective (hence free) if and only if the map $\theta: M\otimes_RM^* \to \Hom_R(M,M)$ taking $x\otimes f$ to the endomorphism $y\mapsto f(y) x$  is surjective; see \cite[Proposition A.1]{AuGo}.  (Surjectivity of $\theta$ provides a dual basis.)  One can assume, by \ref{ch:red-to-tf}, that $M$ is torsionfree, and hence $\Hom_R(M,M)$ is torsionfree as well; now use the fact that a homomorphism $\varphi:U\to V$, with $U$ reflexive and $V$ torsionfree, is an isomorphism if and only if it is an isomorphism in codimension one; see \cite[Lemma 5.11]{LW}. 
\end{subsection}

\section{Local rings of small multiplicity}
\label{sec:small-mult}

We begin with a result, Proposition \ref{prp:2gen-refl}, proved in \cite{Herzing} and attributed to the first and third authors of this paper. We will sketch the proof here, since we have replaced the hypothesis, in \cite{Herzing}, that $R$ is Gorenstein with the weaker one that $I$  is reflexive.  The proof is the same.

\begin{proposition}
\label{prp:2gen-refl}
Let $R$ be a local ring with $\dim R=1$ and $I$ an ideal containing a regular element. If both  $I$ and $I^{*}$ are two-generated, then the $R$-module $I\otimes_{R}I^{*}$ has nonzero torsion.
\end{proposition}

\begin{proof}
Assume to the contrary that $I\otimes_RI^*$ is torsionfree. In the proof, we work with $I^{-1}$ rather than $I^{*}$; keep in mind that they are isomorphic. Also, we identify $I^{-1}$ with the submodule $\{\alpha\in Q(R)\mid \alpha I\subseteq R\}$.

By hypothesis one has $\nu_{R}(I)=2=\nu_R(I^{-1})$. Using prime avoidance, write $I = (a,b)$ and $I^{-1} = (f,g)$, where $a,b,f,g$ are units of $Q(R)$. Let $\varepsilon\colon I^{-1}\to R^{2}$ be the map that assigns $\alpha$ to $\begin{bmatrix} \alpha b\\ -\alpha a\end{bmatrix}$.
It is easy to check that  the sequence
\[
0 \lra I^{-1} \xra{\ \varepsilon\ }  R^{2}\xra{\begin{bmatrix}  a & b  \end{bmatrix}} I \lra 0
\]
is exact.  There is a similar exact sequence
 \[
0 \lra I \xra{\ \sigma\ } R^{2}\xra{\ \tau\ }I^{-1} \lra 0\,,
\]
where $\sigma(c) = \begin{bmatrix} gc \\ -fc \end{bmatrix}$  and  $\tau( \begin{bmatrix} x \\ y \end{bmatrix}) =  fx+gy$\,.  Splicing these  exact sequences together in two ways,
one obtains the following periodic minimal resolution for $I^{-1}$:
\begin{equation}\label{eq:per-res}
\boldsymbol F_\bullet \qquad  \cdots \lra R^2 \xra{\Psi} R^2 \xra{\Phi} R^2 \xra{\Psi} R^2 \xra{\Phi} R^2 \lra 0\,,
\end{equation}
where, with $r=ga, s = gb, t = -f a$, and $u = -fb$, one has
\[
\Phi = \begin{bmatrix} r & s \\ t & u \end{bmatrix} \qquad\text{and}\qquad
\Psi  = \begin{bmatrix} -u & s \\ t & -r \end{bmatrix}\,.
\]
Since $I$ is not invertible, the entries of these matrices are elements of the maximal ideal $\fm$.  Let $J = -fI = (t,u)$.  Since $J$ and $I$ are isomorphic $R$-modules, We see that $J\otimes_RI^{-1}$ is torsionfree.

Tensoring the exact sequence
\[
0 \lra J \lra R \lra R/J \to 0
\]
with $I^{-1}$, we get an injection $\Tor^R_1(R/J,I^{-1}) \hookrightarrow J\otimes_RI^{-1}$.  Since $\Tor^R_1(R/J,I^{-1})$ is killed by the regular element $t$ and $J\otimes_RI^{-1}$ is  torsionfree, we must have 
\begin{equation}
\label{eq:Tor1}
\operatorname{H}_1\big((R/J) \otimes_R \boldsymbol F_\bullet \big) = \Tor^R_1(R/J,I^{-1})  = 0\,.
\end{equation}

The maps in the complex $ (R/J) \otimes_R \boldsymbol F_\bullet$ are given by the
reductions 
\[
\ov \Phi = \begin{bmatrix} \ov  r & \ov  s \\ 0 & 0 \end{bmatrix} \quad\text{and}\quad
\ov \Psi  = \begin{bmatrix}  0 & \ov  s \\ 0 & -\ov  r \end{bmatrix}\,.
\]
of the matrixes $\Phi$ and $\Psi$ modulo the ideal $J$.   We will build a element $\alpha \in \ker \ov \Phi \setminus \im \ov \Psi$, to obtain the desired contradiction.

If $\ov  s = 0$, we can take $\alpha = \left[\begin{smallmatrix}0 \\ 1\end{smallmatrix}\right]$, so assume $\ov  s \ne 0$.
Since $\ov  s$ is nilpotent and $\ov  r$ is not a unit of $\ov  R = R/J$, there is a positive integer $n$ 
such that $\ov  s^n \in \ov  R\ov  r$ but $\ov  s^{n-1}\notin \ov  R\ov  r$.  
Write $\ov  s^n = x\ov  r$, where $x\in \ov  R$.  Then 
$\alpha = \left[\begin{smallmatrix} x \\ -\ov  s^{n-1} \end{smallmatrix}\right]$ belongs to $\ker\ov \Phi$.
If there were elements $y,z\in \ov  R$
with $\ov \Psi \left[\begin{smallmatrix}y \\ z\end{smallmatrix}\right] = \alpha$, one would have
$\ov  r z = \ov  s^{m-1}$, a contradiction.  Thus $\alpha \notin \im\ov \Psi$.
\end{proof}

In the context of two-generated ideals, we should mention the following result of Garcia-Sanchez and Leamer \cite{GaLe}: If a monomial ideal $I$ in a  complete intersection numerical semigroup ring $R$ is two-generated and rigid, then $I$ is principal.

\begin{remark}
\label{rem:erb-bound}
Let $R$ be a Cohen-Macaulay  local ring and $M$ a MCM $R$-module with rank. There is then an inequality
\[
\nu_{R}(M)\le e(R)\rank_{R}(M)\,,
\]
where $e(R)$ denotes the multiplicity of $R$; for a proof of this assertion see, for example, \cite[Lemma 1.6]{HW1}.
\end{remark}

\begin{subsection}{Trace}
\label{ch:end-trace} 
The \emph{trace} of a module $M$ is the image $\Tr_R(M)$ of the canonical map 
\[
M^{*}\otimes_RM \lra R
\]
that assigns an element $f\otimes x$ to $f(x)$.  
\end{subsection}

\begin{proposition}
\label{prp:multiplicity-bound}
Let $R$ be a local ring with $\dim R=1$ and $M$ a nonzero finite $R$-module with rank. If $M\otimes_{R}M^{*}$ is torsionfree, then for each ideal $J\supseteq \Tr_{R}(M)$ there are inequalities
\[
\nu_{R}(J)\le e(R)(1-\rho)+\frac{\nu_{R}(M\otimes_{R}M^{*})}{\nu_{R}(M)^{2}}\le e(R)(1 - \rho + {\rho}^{2})
\]
where $\rho=\rank_{R}(M)/\nu_{R}(M)$. In particular, if $M$ is not free, then $\nu_{R}(J)\le e(R)-1$.
\end{proposition}

\begin{proof}
We may assume $R$ is Cohen-Macaulay, else $M$ is free and hence $\Tr_{R}(M)=R$, and then it is easy to verify that the desired inequalities hold.

Set $\nu=\nu_{R}(M)$, and consider a presentation
\[
0\lra N\lra R^{\nu}\lra M\lra 0
\]
Thus $N$ is the first syzygy module of $M$. Since the ideal $J$ contains $\Tr_{R}(M)$, one has $\Hom_{R}(M,J)=M^{*}$, so applying $\Hom_{R}(-,J)$ to the exact sequence above yields exact sequences
\begin{gather}
\label{eq:split-les1}
0\lra M^{*}\lra J^{\nu} \lra X\lra 0 \\
\label{eq:split-les2}
0 \lra X\lra \Hom_{R}(N,J)\lra \Ext^{1}_{R}(M,J)\lra 0
\end{gather}
Since $M$ has rank, the $R$-module $\Ext^{1}_{R}(M,J)$ is torsion and hence \eqref{eq:split-les2} yields
\[
\rank_{R}(X)=\rank_{R}(\Hom_{R}(N,J)) = \rank_{R}(N) = \nu - \rank_{R}(M)
\]
For the second equality, we have used the fact that $J$ contains $\Tr_{R}(M)$ and hence is of rank one. This computation and Remark~\ref{rem:erb-bound} yield the second inequality below
\begin{align*}
\nu\cdot\nu_{R}(J) = \nu_{R}(J^{\nu}) 
	&\le \nu_{R}(X) + \nu_{R}(M^{*}) \\
	&\le e(R)(\nu-\rank_{R}(M)) + \frac{\nu_{R}(M\otimes_{R}M^{*})}\nu \\
	&\le e(R)(\nu-\rank_{R}(M)) + \frac{ e(R) \rank_{R}(M)^{2}}\nu
\end{align*}
The first inequality is from \eqref{eq:split-les1}. The last inequality is also by Remark~\ref{rem:erb-bound}; it applies as $M\otimes_{R}M^{*}$ is torsionfree. Dividing by $\nu$ yields the stated inequalities.

When $M$ is not free $\rank_{R}M < \nu$ so $0<\rho<1$, which implies $1-\rho+{\rho}^{2}<1$.
\end{proof}

\begin{remark}
\label{rem:analysis}
Proposition~\ref{prp:multiplicity-bound} leads one to consider the parabola  $y = 1- x + x^{2}$,  which has its vertex at the point $(1/2 , 3/4)$. In our context $x=\rank_{R}(M)/\nu_{R}(M)$ and hence, assuming $M$ is not free, one has $0 < x < 1$. In this range, the parabola  satisfies $y < 1$.
\end{remark}

\begin{corollary}
\label{cor:minimal-multiplicity}
Let $(R,\fm)$ be a local ring  with $\dim R=1$. Let $M$ a finite $R$-module with rank and $M\otimes_{R}M^{*}$ torsionfree. If there exist an element $x$ in $R$ and an integer $n$ such that $\Tr_{R}(M)\subseteq \fm^n$ and  $x\fm^n = \fm^{n+1}$,  then $M$ is free. 

In particular, if $R$ has minimal multiplicity, then $M$ is free.
\end{corollary}

\begin{proof} 
We may assume $R$ is Cohen-Macaulay, else $M$ is free. There are equalities
\[
\nu (\fm^n) =\nu (\fm^n/\fm^{n+1}) = \lambda(\fm^n/x\fm^n) = e(\fm^n) = e(R)
\]
where the last two hold as $\fm^n$ is Cohen-Macaulay of rank $1$.  The main conclusion follows from the last statement of
Proposition \ref{prp:multiplicity-bound}. The last statement follows from the fact that if $R$ has
minimal multiplicity, $\fm^2 = x\fm.$
\end{proof}

\begin{lemma}
\label{lem:nu-trace}
Let $R$ be a local ring  and $I$ an ideal containing a regular element. If the $R$-module $I\otimes_{R}I^{*}$ is torsionfree, then the canonical map $I\otimes_{R}I^{*}\to \Tr_{R}(I)$ is bijective and hence
\[
\nu_{R}(I)\nu_{R}(I^{*}) = \nu_{R}(\Tr_{R}(I))\,.
\]
When in addition $\dim R=1$, one has $\nu_{R}(I^{*})\nu_{R}(I)\le e(R)$.
\end{lemma}

\begin{proof}
The map in question is surjective by definition, and generically an isomorphism since $I$ contains a regular element. Since $I\otimes_{R}I^{*}$ is torsionfree it follows that the map is an isomorphism; this justifies the stated equality. It also implies that $\Tr_{R}(I)$ has rank one. When in addition $\dim R=1$ the $R$-module $\Tr_{R}(I)$ is  MCM and \ref{rem:erb-bound} yields the desired inequality.
\end{proof}

\begin{theorem}
\label{thm:HW-general}
Let $R$ be a local ring with $\dim R=1$. Let $I$ be reflexive ideal containing a regular element, and
assume that $I\otimes_{R}I^{*}$ is torsionfree. If $e(R)\le 8$, then the ideal $I$ is  principal.
\end{theorem}

\begin{proof}
Set $\nu :=\nu_{R}(I)$ and $\xi:=\nu_{R}(I^{*})$; suppose $\nu \ge 2$, contrary to the desired conclusion. Since $I$ is reflexive we have $\xi \ge 2$ as well. If $\nu = 2=\xi$, we apply Proposition \ref{prp:2gen-refl}.  Otherwise, we may assume by symmetry  that $\xi \ge 3$.  The equality below is from Lemma~\ref{lem:nu-trace}.
\[
\nu\xi =\nu_{R}(\Tr_{R}(I))\le e(R) (1-\frac 1\nu) + \frac{\nu\xi}{\nu^{2}}
\]
The inequality is from the first inequality in Proposition~\ref{prp:multiplicity-bound} applied to $J=\Tr_{R}(I)$; note that the rank of $I$ is one. This simplifies to an inequality
\[
\xi(\nu+1)\le e(R)\,,
\]
which contradicts the assumption $e(R) \le 8$.
\end{proof}

\begin{remark} 
Theorem \ref{thm:HW-general} extends and was inspired by \cite[Theorem 1.4]{GTTT}, that states, in our notation, that $e > (\nu+1)\xi\geq 6$ if $I$ is not principal. The main new parts of our theorem are the cases of $e = 7,8$.
\end{remark}

\begin{remark} It is important in Theorem \ref{thm:HW-general} that we are tensoring $I$ with
its dual. For example, in \cite[4.3]{HW1} the following example is given: let $R = k[[t^4,t^5,t^6]]$,
$I = (t^4,t^5)R$, $J = (t^4,t^6)R$. Then $R$ is a one-dimensional complete intersection (and so all ideals are reflexive),
$I$ and $J$ are two-generated, but $I\otimes_RJ$ is torsionfree. However, $J$ is not isomorphic
to the dual of $I$. 
\end{remark}

\section{Gorenstein rings}
\label{sec:gorenstein}
From this section on the focus will be on one-dimensional Gorenstein local rings. Let $(R,\fm, k)$ be one such. Given nonzero $R$-submodules $A, B \subseteq Q(R)$ containing a unit of $Q(R)$ we identify $\Hom_R(A,B)$ with
\[
\calH(A,B):=\{\alpha\in Q(R) \mid \alpha A \subseteq B\}.
\]
In particular, $A^* = \calH(A,R)$.  If $A$ is a subring of $Q(R)$ containing $R$, we usually write $\calC_A$ (for ``conductor'') instead of $\calH(A,R)$. These considerations apply to an ideal $I$ containing a regular element; in particular,  we identify $I^{*}$ with $\calH(I,R)$, that is to say, with $I^{-1}$, and $\End_{R}(I)$ with the subring $\calH(I,I)$ of $Q(R)$.

\begin{lemma}
\label{lem:end-trace}
Let $R$ be a Gorenstein ring and  $I$ a proper ideal containing a regular element. Set $S:=\calH(I,I)$. The following statements hold.
\begin{enumerate}[\quad\rm (1)]
\item $\calC_S = \Tr_R(I)$.
 \item $\calH(I^{-1},I^{-1}) = S$.
 \item $\calC_S$ is the largest ideal of $R$ whose endomorphism ring is $S$.
 \item If $A,B$ are $R$-submodules of $Q(R)$ with $A\cong B$, then $\calH(A,A) = \calH(B,B)$.
 \end{enumerate}
\end{lemma}

\begin{proof}
(1) Hom-$\otimes$ adjointness and the fact that $R$ is Gorenstein yield
\[
(\Tr_RI)^*= (I^*I)^*= (I^*\otimes_RI)^* = \Hom_R(I,I^{**}) = S.
\]
Therefore $\Tr_RI = S^*=\calC_S$, as claimed.

(2) For any element $\alpha \in  Q(R)\setminus\{0\}$, it is easy to check that  $\alpha I \subseteq I$ implies $\alpha I^{-1}\subseteq I^{-1}$,
and conversely (by symmetry).  

(3) Observe that $\End_R(S^*) = \End_R(S) = \End_S(S) = S$; this is by (2) and the fact that $S$ is isomorphic to an ideal of $R$.  Suppose next that $J$ is an ideal of $R$ with $\End_R(J)=S$.  Then 
\[
S = \Hom_R(J,J) \subseteq \Hom_R(J,R) = J^*\,.
\]  
Therefore $\calC_S = S^*\supseteq J^{**} = J$.

(4)  There exists an element $\beta\in Q(R)$ such that $B=\beta A$.  Then, for $\alpha \in  Q(R)$ we have $\alpha A\subset A \iff \alpha\beta A \subset \alpha\beta A$, so  $\calH(A,A) = \calH(B,B)$.
\end{proof}

\begin{subsection}{Two-generated ideals and Gorenstein ideals}
\label{ssec:2-gen}
In what follows we use some results on linkage of ideals.  Proper ideals $I,J$ of $R$ are \emph{linked} provided there is regular sequence $\ul{x} = x_1,\dots,x_g$, contained in $I\cap J$ and such that 
\[
J = ((\ul{x}):_RI)\qquad\text{and}\qquad I = ((\ul{x}):_RJ)\,.
\]
For details see, for example, \cite[Section 2]{PS}. When the ring $R$ is Gorenstein there is a correspondence between two-generated ideals and Gorenstein ideals.  Here's how it works \cite[Proposition 2.5]{HU}:

\begin{lemma}
\label{lem:2gen-Gor}
Let $R$ be a Gorenstein local ring and $I$ a height-one unmixed ideal containing a regular element. Then $I$ is two-generated if and only if it is linked to a height one Gorenstein ideal which is not
principal; in this case any height-one  Gorenstein ideal linked to $I$ is isomorphic to $I^{*}$.
\end{lemma}

\begin{proof}
Assume $\nu(I)=2$; thus  $I = (a,b)$, where we can assume that $a$ is a regular element. Set $J:= \big((a):_Rb\big) = \big((a):_RI)$.  Since $I$ is unmixed, by \cite[Section 2]{PS}, one has  $I = \big((a):_RJ\big)$, so $I$ and $J$ are linked.  Moreover $\omega_{R/J} \cong I/(a)$, which is cyclic, so $J$ is a Gorenstein ideal of $R$.  

Conversely, let $J$ be a nonzero height one Gorenstein ideal of $R$ that is not principal. Choose a regular element element $a\in J$ and set $I := \big((a):_RJ\big)$.  Again, $I$ and $J$ are linked via $(a)$, and $I/(a) \cong \omega_{R/J}$, which is cyclic; therefore $I = (a,b)$ for some $b\in I$.

With $I$ two-generated and linked to a Gorenstein ideal $J$  as above, we have 
\[
I^*\cong I^{-1} = \frac{1}{a}J \quad\text{and}\quad J^*\cong J^{-1} = \frac{1}{a}I\,.
\]
This justifies the assertions. We note that if $I$ is a principal ideal, then any ideal linked to $I$ is also principal. 
\end{proof}
\end{subsection}

Proposition~\ref{prp:multiplicity-bound} can be significantly enhanced for ideals in Gorenstein rings.

\begin{proposition}
\label{prp:multiplicity-bound-ideal}
Let $R$ be a Gorenstein local ring with $\dim R=1$. Let $I$ be an ideal in $R$ containing a regular element and such that $I\otimes_{R}I^{*}$ is torsionfree. Set
\[
\delta := e(R) - \max\{\nu_{R}(J)\mid J\supseteq \Tr_{R}(I)\}.
\]
When $I$ is not principal the following inequalities hold:
\begin{enumerate}[\quad\rm(1)]
\item
$\delta  \ge \nu_{R}(I)$ and $\delta \ge \nu_{R}(I^{*})$;
\item
$\delta\ge (e(R) - \nu_{R}(I^{*}))\nu_{R}(I)^{-1}$;
\item
${\delta }^{2}\ge e(R)- \delta  \ge \nu_{R}(I)\nu_{R}(I^{*})\ge 6$.
\end{enumerate}
Moreover the equality ${\delta}^{2}=e(R)-\delta $ holds if and only if $\delta =\nu_{R}(I)=\nu_{R}(I^{*})$.
\end{proposition}

\begin{proof}
Set $e:=e(R)$,  $\nu:=\nu_{R}(I)$ and $\xi:=\nu_{R}(I^{*})$. Since the ring $R$ is Gorenstein and $I$ is reflexive, but not principal, we can assume  that $2\le \nu \le \xi$. 

\medskip

(2) From the first inequality in Proposition~\ref{prp:multiplicity-bound} applied with $M=I$  and  $J$  some ideal containing $\Tr_R(I)$  with $\nu_R(J) = e- \delta $, one gets, after a rearrangement of terms,  the desired inequality:
\[
e \le \nu \delta  + \xi\,.
\]

Since $e- \delta $ is the maximum of the number of generators of an ideal containing $\Tr_{R}(I)$,  Lemma~\ref{lem:nu-trace} yields
\[
e-\delta\ge \nu\xi
\]
which is the second inequality in (3).

\medskip

(1) It suffices to verify $\delta \ge \xi$.  The first inequality below has been justified.
\[
\nu\xi + \delta  \le e \le \nu \delta  + \xi\,;
\]
The second inequality is from the already established part (2). It follows that $(\nu-1)\xi \le (\nu -1)\delta $, and hence that $\xi \le \delta $; recall $\nu\ge 2$.

\medskip

(3) Since $I$ is not principal, Proposition~\ref{prp:2gen-refl} yields that at least one of $\nu$ or $\xi$ is greater than or equal to $3$, and hence that $\nu\xi \ge 6$. It remains to verify that ${\delta }^{2}\ge e-\delta $. The already verified inequalities $\delta \ge \xi$ and $\delta \ge \nu$ yield the first and the last of the following inequalities:
\[
e- \delta  \le e - \xi \le \nu {\delta } \le {\delta }^{2}\,;
\]
the one in the middle is from (2). It remains to observe that $e-\delta ={\delta }^{2}$ holds if and only if both equalities $\delta =\nu$ and $\delta =\xi$ hold.  This completes the proof.
\end{proof}

Here is a first application of Proposition~\ref{prp:multiplicity-bound-ideal}.

\begin{theorem}
\label{thm:HW-gorenstein}
Let $R$ be a local ring with $\dim R=1$. Let $I$ be an ideal in $R$ containing a regular element, and such that $I\otimes_{R}I^{*}$ is torsionfree. If $R$ is Gorenstein of minimal multiplicity, then $I$ is principal.
\end{theorem}

\begin{proof}
Set $\nu :=\nu_{R}(I)$ and $\xi:=\nu_{R}(I^{*})$; suppose $I$ is not principal, so that $\nu \ge 2$.  We use the notation and conclusions of Proposition~\ref{prp:multiplicity-bound-ideal}. The hypothesis on multiplicity means that $e = \edim R +2$. Since $I$ is not free, $\Tr_{R}(I)$ is contained in the maximal ideal of $R$, and hence $\delta \le e - \edim R=2$, which  contradicts ${\delta}^{2}\ge 6$.
\end{proof}

\section{Connection with rigidity}
\label{sec:rigidity}
There is a useful reinterpretation of Conjecture~\ref{conj:HW} in term of rigidity: A module $M$ over a ring $R$ is \emph{rigid} if $\Ext^{1}_{R}(M,M)=0$.  The following result is well-known; we record it for ease of reference. We write $\Omega^n_RM$ for the $n^{\text{th}}$ syzygy (defined only up to projective summands) of $M$; when $R$ is Gorenstein and $M$ is MCM, this makes sense for all $n\in \ZZ$.

\begin{proposition}
\label{prp:tf=rigid}
Let $R$ be a Gorenstein local ring with $\dim R=1$ and $M$ a finite  torsionfree $R$-module with rank. The following conditions are equivalent:
\begin{enumerate}[\quad\rm(1)]
\item
$M\otimes_{R}M^{*}$ is torsionfree;
\item
$\Tor^{R}_{1}(\Omega^{-1}M,M^{*})=0$;
\item
$M$ is rigid.
\end{enumerate}
\end{proposition}

\begin{proof}
For $R$ as in the statement, a finitely generated module is torsionfree if and only if it is MCM. Thus (1) $\iff$ (3) is a special case of \cite[Theorem~5.9]{HuJo}.

(1) $\iff$ (2) By definition of $\Omega^{-1}M$, there is an exact sequence
\[
0\lra M \lra F\lra \Omega^{-1}M\lra 0
\]
where $F$ is finite free. Tensoring with $M^{*}$ yields an exact sequence
\[
0\lra \Tor^{R}_{1}(\Omega^{-1}M,M^{*}) \lra M\otimes_{R}M^{*}\xra{\ f\ } F\otimes_{R}M^{*}
\]
Since $M^{*}$ is torsionfree, so is $F\otimes_{R}M^{*}$, and since $M$ has rank so does $M^{*}$, and hence $\Tor^{R}_{1}(\Omega^{-1}M,M^{*})$ is torsion. It follows that the latter module is zero precisely when $M\otimes_{R}M^{*}$ is torsionfree.
\end{proof}

In view of the preceding result Conjecture~\ref{conj:HW} translates to: Over a Gorenstein local ring of dimension one, any finitely generated rigid module is free.  In the remainder of this section we once again focus on the case of an ideal. The new ingredient will be its endomorphism ring.  Given an ideal $I$ containing a regular element,  when needed its endomorphism ring $\End_{R}(I)$ will be viewed as a subring of the integral closure of $R$; see the remarks at the beginning of Section~\ref{sec:gorenstein}.

The observation below is due to Bass~\cite[Section 7]{Bass}. 

\begin{remark}
\label{rem:ubiquity}
Let $R$ be a Gorenstein ring with $\dim R=1$ and $I$ an ideal containing a regular element. If $\End_{R}(I)=R$, then $I$ is principal. 

Indeed,  we  assume $R$ is not regular, that is to say, $\fm$ is not principal. Then the endomorphism ring $E$ of $\fm$ is equal to $\fm^*$.   Similarly, if $I$ is not principal, then $I^* = \Hom_R(I,\fm)$, which is an $E$-module. Therefore $I = I^{**}$ is an $E$-module too, so $E\subseteq R$. Applying $(-)^*$ to the exact sequence
\[
0\to \fm \to R \to k \to 0
\]
we see that $E/R\cong \Ext^1_R(k,R) \cong k$, so $S$ is a proper extension of $R$.
\end{remark}

The next observation is contained in~\cite[Theorem~6.4]{Lindo}, due to Lindo. A simpler proof is available in our special case.

\begin{proposition}
\label{prop:ubiquity} 
Let $R$ be a Gorenstein local ring with $\dim R=1$ and $I$ an ideal containing a regular element.  If $I$ is rigid and the ring $\End_{R}(I)$ is Gorenstein, then $I$ is principal.
\end{proposition}

\begin{proof} 
Set $S:=\End_{R}(I)$; this is commutative, for it can be viewed as a subring of the integral closure of $R$.  Remark~\ref{rem:ubiquity} implies that $I$ is a principal ideal of $S$, whence $I\cong S$: Even if $S$ is not local, we simply apply the remark to each local ring of $S$, and then use the Chinese Remainder Theorem.

The canonical surjection $\pi\colon I\otimes_RI^* \twoheadrightarrow I\otimes_SI^*$ is bijective,  because the  injection $I\otimes_RI^* \to Q(R)\otimes_RI\otimes_RI^*$ factors through $\pi$.  It follows that 
\[
\nu_R(I)\nu_R(I^*) = \nu_R(I\otimes_SI^*) = \nu_R(S\otimes_SI^*) = \nu_R(I^*)\,.
\]
Canceling the factor $\nu_R(I^*)$, we get the equation $\nu_R(I) = 1$.
\end{proof}

\begin{remark}
It follows from Proposition~\ref{prop:ubiquity} that if  all rings between $R$ and its integral closure are Gorenstein then there are no non-principal rigid ideals---but much more is known in
this case; see \cite{Bass}.
\end{remark}

\begin{remark}
Let $I$ be an ideal in $R$ that contains a regular element and is rigid; set $S:=\End_{R}(I)$. Applying $\Hom_{R}(-,R)$ and $\Hom_{R}(I,-)$, respectively, to the exact sequence $0\to I\to R\to R/I\to 0$, induces the rows of the following commutative diagram of $R$-modules:

\begin{equation}
\label{eq:CD}
\xymatrix{ 
0 \ar@{->}[r] & R \ar@{->}[r] \ar@{->}[d] & I^{*} \ar@{->}[r] \ar@{=}[d] 
		& \Ext^{1}_{R}(R/I,R) \ar@{->}[r] \ar@{->}[d]^{\gamma} & 0 \\
0 \ar@{->}[r] & S \ar@{->}[r] &  I^{*} \ar@{->}[r]
		 & \Hom_{R}(I,R/I) \ar@{->}[r] & \Ext^1_R(I,I)=0  }
\end{equation}
The vertical map on the left is the canonical homothety; it is one-to-one because $I$ contains a regular element. The map $\gamma$ is induced by the other two vertical arrows. 

The Snake Lemma applied to the diagram above implies that $\gamma$ is surjective and identifies $\ker\gamma$ with $S/R$. One thus gets an exact sequence
\begin{equation}
\label{eq:exact1}
0\lra S/R  \lra \Ext^{1}_{R}(R/I,R) \lra \Hom_{R/I}(I/I^{2},R/I)\lra 0\,.
\end{equation}
\end{remark}

Suppose now that $(R,\fm,k)$ is one dimensional, Gorenstein.  The canonical module of $R/I$ is given by
\begin{equation}
\label{eq:can}
\omega_{R/I} = \Ext^1_R(R/I,R)\,;
\end{equation}
see, for example, \cite[Theorem 3.3.7]{Bruns/Herzog:1993}. Moreover, we have
\begin{equation}
\label{eq:rhom}
\begin{aligned}
\Hom_{R/I}(I/I^2\otimes_{R/I}\omega_{R/I},\omega_{R/I})
	& =\Hom_{R/I}(I/I^2,\Hom_{R/I}(\omega_{R/I},\omega_{R/I})) \\
	& = \Hom_{R/I}(I/I^2,R/I)\,.
\end{aligned}
\end{equation}
Therefore the exact sequence \eqref{eq:exact1} becomes
\begin{equation}
\label{eq:3term1}
0\lra S/R \lra \omega_{R/I} \lra \dual{((I/I^2)\otimes_{R/I}\omega_{R/I})} \lra 0 \,,
\end{equation}
where $\dual{(-)}$ denotes the $\omega_{R/I}$-dual $\Hom_{R/I}(-,\omega_{R/I})$.

Denote the length of an $R$-module $M$ by $\lambda_R(M)$.

\begin{proposition}
\label{prop:Ext-length} 
Let $R$ be a Gorenstein local ring with $\dim R=1$ and $I$ an ideal containing a regular element.  
We have the following equality:
\begin{equation}
\label{eq:Ext-length}
\lambda_R(\Ext_R^1(I,I)) + \lambda_R(R/I) = \lambda_R(S/R)
+\lambda_R((I/I^2)\otimes_{R/I} \omega_{R/I})\,.
\end{equation}
In particular, the ideal $I$ is rigid if and only if
\begin{align*}
\lambda_R(R/I) - \lambda_R(S/R) & = \lambda_R((I/I^2)\otimes_{R/I}
\omega_{R/I})\, 
\intertext{if and only if}
\lambda\left((\Tr_RI)/I\right) &= \lambda_R\big((I/I^2)\otimes_R\omega_{R/I}\big) \,.
\end{align*}
\end{proposition}

\begin{proof}
As $R$ is Gorenstein and $I$ contains a regular element $\Ext^{1}_{R}(I,R)=0$ holds. Thus, applying  $\Hom_R(I,-)$  to the exact sequence $0\to I \to R\to R/I\to 0$ yields an exact sequence
\[
0\to S \to I^* \to \Hom_{R/I}(I/I^2,R/I)\to \Ext^1_R(I,I)\to 0\,.
\]
Given \eqref{eq:rhom} this translates to an exact sequence
\begin{equation}
\label{eq:exact}
0\to S/R \to I^*/R \to \dual{((I/I^2)\otimes_{R/I}\omega_{R/I})} \to
\Ext^1_R(I,I) \to 0\,.
\end{equation}
The top row of the commutative diagram \eqref{eq:CD} shows  that $I^*/R\cong \omega_{R/I}$.  Thus
\[
\lambda_R(I^*/R) = \lambda_R(R/I)\,.
\]
Also, by duality, we have
\[
\lambda_R(\dual{((I/I^2)\otimes_{R/I}\omega)}) = \lambda_R
((I/I^2)\otimes_{R/I}\omega)\,.
\]
Using these equalities and counting lengths in \eqref{eq:exact} yields \eqref{eq:Ext-length}.
\end{proof}

\begin{corollary}
\label{cor:funny}
Let $R$ be a Gorenstein local ring $R$ with $\dim R=1$ and $I=(a,b)$ an ideal such that $a$ and $b$ are regular elements.  Then $I$ is rigid if and only if  
\[
((a):b)\cap ((b):a) = ((a):b)((b):a)\,.
\]
\end{corollary}

\begin{proof}
\pushQED{\qed}
Set $J:=((a):b)$ so that $I$ and $J$ are linked via  $(a)$. One checks that
\[
bJ = b\big((a):_Rb\big) = (a) \cap (b) = a\big((b):a\big)\,.
\]
Therefore
\[
\Tr_RJ = \Tr_RI = (a,b)I^{-1} = a\frac{1}{a}J + b\frac{1}{a}J = \big((a):_Rb \big)+ \big((b):_Ra\big)\,.
\]
Now, by Proposition \ref{prop:Ext-length} $J$ is rigid if and only if
\[
\lambda_R (\frac{\Tr_RJ}{J}) = \lambda_R(\frac{J}{J^2})\,,
\]
if and only if
\[
\lambda_R(\frac{\big((a):_Rb\big)+\big((b):_Ra\big)}{\big((a):_Rb\big)}) = 
\lambda_R(\frac{\big((a):_Rb\big)}{\big((a):_Rb\big)^2})\,,
\]
if and only if
\[
\lambda_R(\frac{\big((b):_Ra\big)}{\big((a):_Rb\big)\cap\big((b):_Ra\big)})=
\lambda_R(\frac{b\big((a):_Rb\big)}{b\big((a):_Rb\big)^2})\,,
\]
if and only if
\[
\lambda_R(\frac{\big((b):_Ra\big)}{\big((a):_Rb\big)\cap\big((b):_Ra\big)})
=\lambda_R(\frac{a\big((b):_Ra\big)}{a\big((b):_Ra\big)\big((a):_Rb\big)})\,,
\]
if and only if 
\[
\lambda_R(\frac{\big((b):_Ra\big)}{\big((a):_Rb\big)\cap\big((b):_Ra\big)})
=\lambda_R(\frac{\big((b):_Ra\big)}{\big((b):_Ra\big)\big((a):_Rb\big)})\,,
\]
if and only if 
\[
\big((a):_Rb\big)\cap\big((b):_Ra\big) = \big((a):_Rb\big)\big((b):_Ra\big)\,.\qedhere
\]
\end{proof}

\begin{remark} 
Corollary \ref{cor:funny} is rather remarkable, for it shows that the existence of a two-generated rigid ideal is \emph{equivalent} to the existence of  isomorphic ideals $I$ and $J$ such that $R/I$ and $R/J$ are Gorenstein and $\Tor_1^R(R/I,R/J) = 0$.  It seems  unlikely that such ideals exist.
\end{remark}

The preceding results highlight the importance of studying endomorphism rings of rigid ideals.  The result below, though not strictly necessary in the sequel, anticipates the key role linkage plays in the next section.

\begin{lemma}
\label{lem:link-endo}  
Let $I$ and $J$ be ideals of a Noetherian ring $S$. Assume each contains a regular element,  and identify their endomorphism
rings $A$ and $B$, respectively, with subrings of the integral closure of $S$. If $I$ and $J$ are linked, then $A= B$.
\end{lemma}

\begin{proof}  
We have $J = ((\ul{x}:_SI)$ and $I= ((\ul{x}:_SJ)$, where $\ul{x}$ is a regular sequence.  Suppose $t\in A$, that is, $t$ is in the integral closure of $S$ and $tI\subseteq I$.  If $z\in J$, we have $tzI = z(tI) \subseteq zI \subseteq (\ul{x})$. This shows that $tz \in ((\ul{x}):_RI) = J$; thus $t\in B$.  We have shown that $A\subseteq B$.  By symmetry, the two endomorphism rings coincide.
\end{proof}

\section{Smoothable ideals}
\label{sec:smoothable}
From now on the focus will be on complete intersection rings. In this section we verify Conjecture~\ref{conj:HW} for ideals  that deform to  smoothable ideals. To this end we recollect some basic properties of  twisted conormal modules.

\subsection{The twisted conormal module}
\label{ssec:conormal}
Consider a pair $(Q,J)$ where $Q$ is a Gorenstein local ring and $J\subset Q$ an ideal such that $Q/J$ is Cohen-Macaulay. The hypotheses on $Q$ and $J$ imply that $\Ext^{d}_{Q}(Q/J,Q)$,  where $d=\height I$,  is the canonical module of $Q/J$. Motivated by Proposition~\ref{prop:Ext-length}, we consider the $Q/J$-module
\[
\con QJ :=  (J/J^{2})\otimes_{Q/J} \omega_{Q/J}\,,
\]
which Buchweitz~\cite{Buchweitz:1981} calls the \emph{twisted conormal module}. We are particularly interested in the situation when $Q/J$ has finite length, and satisfies an inequality
\begin{equation}
\label{eq:conj-conormal}
\lambda_{Q}(\con QJ) \geqslant (\dim Q)\lambda_{Q}(Q/J)\,.
\end{equation}
In this case $\omega_{Q/J}$ is the injective hull of $k$ as a module over $Q/J$. Thus, by Matlis duality, the preceding inequality is equivalent to 
\[
\lambda_{Q}(\Hom_{Q/J}(J/J^{2},Q/J)) \geqslant (\dim Q)\lambda_{Q}(Q/J)\,.
\]
While such an inequality does not hold in general---see Example~\ref{ex:EU}---the main agenda in this section is to identify various classes of ideals for which it does. We begin with the following simple observation.

\begin{remark}
\label{rem:conj-ci}
In the context of \eqref{eq:conj-conormal}, since $Q/J$ has finite length one has 
\[
\nu_{Q/J}(J/J^{2}) = \nu_{Q}(J)\ge \dim Q
\]
Thus when $J$ is generated by a regular sequence, the $Q/J$-module $J/J^{2}$, and hence also its dual, is free and so \eqref{eq:conj-conormal} holds; in fact, we get an equality.  This observation can be extended to cover the case when $Q$ is regular and $J$ can be deformed to a ideal that is generically complete intersection; see Proposition~\ref{prop:conj-deforms}.  
\end{remark}

Our interest in \eqref{eq:conj-conormal} is explained by the next result.

\begin{corollary}
\label{cor:rigid-length} 
Let $R$ be a Gorenstein local ring with $\dim R=1$ and $I$ a rigid ideal containing a regular element. If $\con RI$  satisfies the inequality in \eqref{eq:conj-conormal}, then $I$ is principal. 
\end{corollary}

\begin{proof}
If $I$ is rigid and $\con RI$  satisfies the inequality in \eqref{eq:conj-conormal}, it follows from Proposition~\ref{prop:Ext-length} that $R=\End_{R}(I)$, and so Remark~\ref{rem:ubiquity} yields that $I$ is principal.
\end{proof}

A special case of this result seems worth highlighting. 

\begin{corollary}
Let $R$ be a Gorenstein local ring with $\dim R=1$ and $I$ a rigid ideal, containing a regular element. If the $R/I$-module $\Hom_{R/I}(I/I^{2},R/I)$ contains a free summand, then $I$ is principal.
\end{corollary}

\begin{proof}
The hypothesis implies the inequality below
\[
\lambda_{R}\con RI = \lambda_{R}\Hom_{R/I}(I/I^{2},R/I)\ge \lambda_{R}(R/I) 
\]
the equality is by Matlis duality. It remains to recall Corollary~\ref{cor:rigid-length}.
\end{proof}

\subsection{Linkage}
\label{ch:linkage}
Ideals $I,J$ in a Gorenstein ring are said to be \emph{in the same linkage class} provided there is a sequence $I = I_0, I_1,\dots, I_n = J$ of
ideals such that for each $0\le r \le n-1$ the ideals $I_r$ and $I_{r+1}$ are linked, in the sense of \ref{ssec:2-gen}. An ideal is \emph{licci} provided it is in the linkage class of a complete intersection ideal, that is to say one that can be generated by a regular sequence.

The result below, attributed to Buchweitz in \cite{Her},  asserts that for ideals of finite projective dimension the validity of \eqref{eq:conj-conormal} is invariant under linkage.

\begin{proposition}
\label{prop:BuchHer}  
Let $(R,I)$ be a pair as in \ref{ssec:conormal} with $R/I$ of finite length and of finite projective dimension. If $J\subset R$ is in the linkage class of $I$, then
\[
\lambda_{R}(\con RI) - (\dim R)\lambda_{R}(R/I) =
\lambda_{R}(\con RJ) - (\dim R)\lambda_{R}(R/J)\,.
\]
\end{proposition}

\begin{proof}[Remarks in place of a proof] 
See the proof of \cite[Theorem (2.4)]{Her}, in particular, the second display after the statement of Theorem (2.5). Also, it is shown in the proof of \cite[Theorem (2.2)]{Her} that finite projective dimension is preserved by linkage, so that all ideals in the chain of links from $I$ to $J$ have finite projective dimension.  Moreover, all of them are $\fm$-primary and hence have grade equal to the common lengths of the regular sequences forming the links.
\end{proof}

%
%

\subsection{Deformations}
\label{ssec:deformations}
The pair $(R,I)$ \emph{deforms} to $(Q,J)$, where $Q$ is a local ring and $J$ is an ideal in $Q$, if there exists a sequence $\ul x$ in $Q$ that is regular both on $Q$ and $Q/J$, and there are isomorphisms 
\[
\frac Q{(\ul x)} \cong R \quad\text{and}\quad \frac Q{J+(\ul x)}\cong \frac RI\,.
\]
We say $R$ \emph{deforms} to $Q$ if $(R,0)$ deforms to $(Q,0)$, that is to say, $R\cong Q/\ul x$ for a $Q$-regular sequence $\ul x$. In these cases $R$ is Gorenstein if and only if $Q$ is Gorenstein, and $R/I$ is CM if and only if $Q/J$ is CM; thus $\con RI$ is defined if and only if $\con QJ$ is defined.  These are remarks will be used without further comment.

The proof of the following result is standard, so is omitted. 

\begin{lemma}
\label{lem:con-deforms}
When $(R,I)$ deforms to $(Q,J)$ there is a natural isomorphism of $R$-modules $\con RI \cong R\otimes_{Q}\con QJ\,.$\qed
\end{lemma}

\subsection{Smoothability}
\label{def:smoothable} 
As before, let $R$ be a Gorenstein ring and $I\subset R$ an ideal. We say that $I$ is \emph{smoothable} provided the pair $(R,I)$ deforms, in the sense of \ref{ssec:deformations}, to $(Q,J)$ where $J$ is generically complete intersection, that is to say, for each prime $\fq$ in $Q$ associated to $J$, the ideal $J_{\fq}$ is generated by a regular sequence.  Observe that the ring $Q$ is also Gorenstein.  When in addition, $Q$ is a quotient of a regular local ring, the pair $(Q,J)$, and hence also the pair $(R,I)$, can be further deformed to a pair $(Q',J')$, where $J'$ is reduced; see \cite[Theorem~3.10]{HU2}. This remark reconciles our notion of smoothability with the one introduced in \cite[Definition~4.2]{HU}.

\begin{proposition}
\label{prop:conj-deforms}
Let $(R,I)$ be a pair as in \ref{ssec:conormal} with $R/I$ of finite length. Then $\con RI$ satisfies the inequality in \eqref{eq:conj-conormal} when any of the following conditions holds:
\begin{enumerate}[\quad\rm(1)]
\item
$R$ deforms to $Q$ and with $K$ the inverse image of $I$ under the quotient map $Q\twoheadrightarrow R$, the $Q/K$-module $\con QK$ satisfies \eqref{eq:conj-conormal};
\item
$(R,I)$ is smoothable;
\item
$I$  is licci.
\end{enumerate}
\end{proposition}

\begin{proof}
In either case $R\cong Q/(\ul x)$, where $\ul x$ is a $Q$-regular sequence, say of length $c$.

(1) Let $\omega$ denote the canonical module of $S:=R/I$. Since $I = K/(\ul x)$ there is an exact sequence
 \[
 \frac{K^2+(\ul x)}{K^2}\otimes_{S}\omega \lra \frac{K}{K^2}\otimes_{S}\omega \lra \frac{I}{I^2}\otimes_{S}\omega \lra 0\,,
 \]
of $S$-modules. Noting that $S^{c}$ maps onto $(K^2+(\ul x))/K^2$ and that $Q/K=R/I$, one gets an exact sequence
of $S$-modules:
\[
\omega^{c}\lra \con QK \lra \con RI\lra 0\,.
\]
This yields the first of the following (in)equalities
\begin{align*}
\lambda  ( \con RI)
	&\ge \lambda (\con QK)  - c \lambda (\omega)\\
	&\ge (\dim Q)\lambda(S) - c \lambda (S) \\
	&= (\dim Q - c) \lambda(S) \\
	&=(\dim R) \lambda (S)\,.
 \end{align*}
The second inequality is by our hypothesis that $\con QK$ satisfies \eqref{eq:conj-conormal}, and the fact that $\lambda(\omega)=\lambda(S)$, which is a consequence of Matlis duality. The equalities are clear. This gives the desired conclusion.

(2)  By hypothesis, $(R,I)$ deforms to $(Q,J)$ with $J$ generically complete intersection. In particular, both $(Q/J)$-modules $J/J^{2}$ and $\omega_{Q/J}$ have rank, equal to $\height J$ and $1$, respectively. Therefore one gets equalities
\[
\rank_{Q/J}(\con QJ) = \dim Q - \dim(Q/J) = \dim R - \dim (R/I)=\dim R\,,
\]
where the first and the last one hold because $Q$ and $R$ are Gorenstein, and the second one holds because $(Q,J)$ is a deformation of $(R,I)$. This computation will be used further below. For the remainder of the proof, we employ the multiplicity symbol $e(\ul x; -)$ on $Q/J$-modules. This is valid since $\ul x$ is a (regular) system of parameters for $Q/J$; recall that $R/I$ has finite length. From Lemma~\ref{lem:con-deforms} and \cite[Theorem 4.7.10(1)]{Bruns/Herzog:1993} one gets the first inequality below.
\begin{align*}
\lambda_{Q} (\con RI) 
	& \ge e(\ul x, \con QJ) \\
	& = e(\ul x, Q/J)\rank_{Q/J}(\con QJ) \\
	&=e(\ul x, Q/J) (\dim R)\\
	&=\lambda (R/I)(\dim R)
\end{align*}
The first equality is by \cite[Corollary 4.7.9]{Bruns/Herzog:1993}, which applies since $\con QJ$ has rank, equal to $\dim R$. The other ones are clear.

(3) This follows from Proposition~\ref{prop:BuchHer} and Remark~\ref{rem:conj-ci}.
\end{proof}

\begin{corollary}
\label{cor:embdim-three} 
Let $R$ be a complete intersection local ring of embedding dimension three with $\dim R=1$, and let $I$ a rigid ideal containing a regular element such that $R/I$ is Gorenstein. Then $I$ is principal.
\end{corollary}

\begin{proof} 
We may assume that $R$ is complete.  By hypothesis, one can find a regular local ring $Q$ of dimension three mapping onto $R$. The preimage $K$ in $Q$ of the ideal $I$ is also of codimension three, and $Q/K$ is Gorenstein. An argument due to J.~Watanabe, implicit in the proof of \cite[Theorem]{Wa}, shows $K$ is licci.  Thus $\con RI$ satisfies \eqref{eq:conj-conormal}, by Proposition \ref{prop:conj-deforms}(3). It remains to recall Corollary~\ref{cor:rigid-length}.
\end{proof}

\begin{theorem}
\label{thm:smoothable-rigid} 
Let $R$ be a complete intersection local ring with $\dim R=1$, and $I$ an ideal containing a regular element.  Let $\pi\colon Q\twoheadrightarrow R$ be a surjective local homomorphism with $Q$ a regular local ring and set $K:=\pi^{-1}(I)$.

If $I$ is rigid and $\con QK$ satisfies the inequality in \eqref{eq:conj-conormal}, in particular, if $(Q,K)$ is licci or smoothable, then $I$ is principal.
 \end{theorem}

 \begin{proof}
 By Proposition~\ref{prop:conj-deforms}(1) the $R/I$-module $\con RI$ also satisfies \eqref{eq:conj-conormal}, so the desired result is a consequence of Corollary \ref{cor:rigid-length}.
 \end{proof}

Here is an application of the preceding results.

\begin{theorem}
\label{thm:ci-smoothable}
Let $R$ be a complete intersection local ring with $\dim R=1$, and $I$ a rigid ideal containing a regular element.
If $e(R)\le 10$, then $I$ is principal.
\end{theorem}

\begin{proof}
Given Theorem~\ref{thm:HW-general} we only have to treat the case when $e(R)$ is $9$ or $10$. We assume that $I$ is not principal, and hence also that $\edim R \ge 4$; see Corollary~\ref{cor:embdim-three}. By passing to  completions, if necessary, we can ensure that $R$ is of the form $Q/(f_{1},\dots,f_{c})$, with $Q$ a regular local ring  and $\boldsymbol f$  a $Q$-regular sequence contained in $\fn^{2}$, where $\fn$ is the maximal ideal of $Q$; note that $c=\edim Q-1$.  Let $K$ be the inverse image of $I$ under the surjection $Q\to R$. We will prove that  $K$ is licci, contradicting Theorem~\ref{thm:smoothable-rigid}. To that end note that if  $d_{1},\dots,d_{c}$ are nonnegative integers such that $f_{i}\in \fn^{d_{i}}\setminus \fn^{d_{i}+1}$, then it follows from  \cite[Chapter VIII, \S7, Proposition 7]{Bourbaki:2006} that
\[
e(R) \ge d_{1}\cdots d_{c}\,.
\]
Since $d_{i}\ge 2$ and  $c=\edim R-1\ge 3$, the hypothesis  $e(R)\le 10$  implies $c=3$, so that $\edim R=4$, and that $d_{i}=2$ for each $i$.  In particular, the number of generators of $\fm^2$ is 7; here  $\fm$ is the maximal ideal of $R$.  This will be used later on.

We consider the integer
\[
\delta := e(R) - \max\{\nu_{R}(J)\mid J\supseteq \Tr_{R}(I)\}.
\]
Since $I$ is not principal, from Proposition~\ref{prp:multiplicity-bound-ideal} one gets inequalities
\[
6\le \nu_{R}(I) \nu_{R}(I^{*}) \le e(R) - \delta \le \delta^{2}\,,
\]
Since $e(R)\le 10$, these imply that $\delta$ is one of $\{3,4\}$, and in either case one has equality on the left, so that $\nu_{R}(I)=3$ and $\nu_{R}(I^{*})=2$, or vice versa. We can assume without loss of generality that $\nu_{R}(I^{*})=2$, so  $I$ is Gorenstein, by Lemma~\ref{lem:2gen-Gor}.

Assume $e(R)=9$, so that $\delta =3$.  Since the number of generators of $\fm^2$ is 7 and $\delta=3$, it follows from the definition of $\delta$ that $\Tr_{I}(R)$ contains an element not in $\fm^2$.   Write $I^{*} = (a,b)$ for elements $a,b\in R$. Then $I = ((a):b)$, and the trace ideal is $((a):b) + ((b):a)$.  Since the trace ideal is not contained in $\fm^2$, without loss of generality we may assume that $((a):b)$ is not contained in $\fm^2$. Choose an element $c\in ((a):b)$ such that $c\notin \fm^2$.  Write $cb = da$. Then $I = ((a):b) = ((c):d)$, so after changing notation we may assume that $a\notin \fm^2$.  We have shown that $Q/K\cong R/I$ is Gorenstein and that $K$ is not contained in the square of the maximal ideal of $Q$. Lift $a$ to $Q$ and by abuse of notation, call that element $a$ as well. Then $Q/(a)$ is a three-dimensional regular local ring, and $K/(a)$ is a grade-three Gorenstein ideal. As noted in the proof of Corollary~\ref{cor:embdim-three}, such an ideal is licci.  It  follows that $K$ is also licci: do the links in $Q/(a)$ and  lift back to $Q$ by taking the same regular sequences and  throwing in $a$. This is the desired result.

Assume  $e(R)=10$. When $\delta=3$ inequality (2) in Proposition~\ref{prp:multiplicity-bound-ideal} is violated for
\[
\nu_{R}(I^{*})\delta + \nu_{R}(I) = 2(3)+3 = 9 \,.
\]
Therefore $\delta=4$. As in the case of $e(R)=9$ and $\delta=3$ we deduce that $\Tr_{R}(I)$ is not in $\fm^2$, so that $K$ is licci. 
\end{proof}

\begin{example} 
\label{ex:EU}
D. Eisenbud and B. Ulrich provided us with examples that show that inequality \eqref{eq:conj-conormal} does not hold in general, even when $Q$ is a regular ring and $J$ is a Gorenstein ideal. 

One such example is obtained by letting $Q$ be the polynomial ring (say over $\CC$) in six variables, and $J$ the Gorenstein ideal
with dual socle element  a general cubic. The Hilbert function $Q/J$ is then $1, 6, 6, 1$. The length of $J/J^2$, and hence also of $\con QJ$, is $76$ so that
\[
\lambda_{Q}(\con QJ) -  (\dim Q)\lambda_{Q}(Q/J) =  76 -  6\cdot 14 = -8\,.
\]
This example  is part of a family appearing in the work of J. Emsalem and A. Iarrobino \cite{EI}, in which they study smoothing of algebras via their tangent spaces.
\end{example}

\bibliographystyle{plain}

\begin{thebibliography}{10}

\bibitem{Au}
M.~Auslander,
\newblock Modules over unramified regular local rings,
\newblock \emph{Illinois J. Math.} \textbf{5} (1961), 631--647.

\bibitem{AuGo}
M.~Auslander and O.~Goldman,
\newblock Maximal orders.
\newblock \emph{Trans. Amer. Math. Soc.} \textbf{97}, 1--24.

\bibitem{Bass}
H. Bass
\newblock On the ubiquity of Gorenstein rings,
\newblock \emph{ Math. Z.}  \textbf{82} (1963) 8--28.

\bibitem{Bourbaki:2006}
N.~Bourbaki,
\newblock \emph{Commutative Algebra}, {C}hapters 8 and 9, Elements of Mathematics,
\newblock Springer-Verlag, Berlin, 2006.

\bibitem{Bruns/Herzog:1993}
W.~Bruns and J.~Herzog,
\newblock \emph{Cohen-{M}acaulay rings}, Cambridge Studies in Advanced Mathematics \textbf{39},
\newblock Cambridge University Press, Cambridge, 1993.

\bibitem{Buchweitz:1981}
R.-O.~Buchweitz,
\newblock \emph{Contributions \`a la th\'eorie des singularit\'es : D\'eformations de Diagrammes, D\'eploiements et Singularit\'es tr\`es rigides, Liaison alg\'ebrique},
\newblock Thesis, University of Paris, 1981; available at \texttt{https://tspace.library.utoronto.ca/handle/1807/16684}

\bibitem{Cel}
O.~Celikbas,
\newblock Vanishing of $\Tor$ over complete intersections,
\newblock \emph{J. Commut. Algebra} \textbf{3} (2011), 169--206.

\bibitem{CW}
O.~Celikbas and R.~Wiegand,
\newblock Vanishing of Tor, and why we care about it,
\newblock \emph{J. Pure Appl. Algebra} \textbf{219}, 429--428.

\bibitem{EI}
A. Iarrobino and J. Emsalem,
\newblock Some zero-dimensional generic singularities; finite algebras having small tangent space
\newblock \emph{Compositio Math.} \textbf{36} (1978), 145--188.

\bibitem{GaLe}
Garc\'\i a-S\'anchez and M. Leamer,
\newblock Huneke-Wiegand Conjecture for complete intersection numerical semigroup rings,
\newblock \emph{J. Algebra} \textbf{391} (2013), 114--124.

\bibitem{GTTT} 
S. Goto, R. Takahashi, N. Taniguchi, H. Le Truong, 
\newblock  Huneke-Wiegand conjecture and change of rings,
\newblock \emph{ J. Algebra}  \textbf{422} (2015), 33--52.

\bibitem{Herzing}
K.~Herzinger,
\newblock The number of generators for an ideal and its dual in a numerical semigroup
\newblock \emph{Comm. Algebra} \textbf{27} (1999), 4673--4690.

\bibitem{Her}
J.~Herzog,
\newblock Homological properties of the modules of differentials,
\newblock \emph{Soc. Bras. Mat., Rio de Janeiro} (1981), 33--64.

\bibitem{HU2}
C.~Huneke and B.~Ulrich,
\newblock Algebraic linkage,
\newblock \emph{Duke Math. J.} \textbf{56} (1988), 415--429.

\bibitem{HU}
C.~Huneke and B.~Ulrich,
\newblock The structure of linkage,
\newblock \emph{Ann. Math.} \textbf{126} (1987), 277--334.

\bibitem{HW1}
C.~Huneke and R.~Wiegand,
\newblock Tensor products of modules and the rigidity of Tor,
\newblock \emph{Math. Ann.} \textbf{299} (1994), 449--476; \texttt{Correction:} \emph{Math. Ann.} \textbf{338} (2007), 291--293.

\bibitem{HuJo}
C.~Huneke and D.~A. Jorgensen,
\newblock Symmetry in the vanishing of Ext over Gorenstein rings,
\newblock \emph{Math. Scand.} \textbf{93} (2003), 161--184.

\bibitem{Kn}
H.~Kn\"orrer,
\newblock Cohen-Macaulay modules on hypersurface singularities, I,
\newblock \emph{Invent. Math.} \textbf{88}, 153--164.

\bibitem{LW}
G.~Leuschke and R.~Wiegand,
\newblock  \emph{Cohen-Macaulay Representations},
\newblock Mathematical Surveys and Monographs \textbf{181}, Amer. Math. Soc., Providence, RI 2012.

\bibitem{Lindo}
H.~Lindo,
\newblock Trace ideals and centers of endomorphism rings of modules over commutative rings,
\newblock \emph{J. Algebra} \textbf{482} (2017), 102--130.

\bibitem{PS}
C. Peskine and L. Szpiro, 
\newblock Liaison des vari\'et\'es alg\'ebriques. I, 
\newblock \emph{Invent. Math.},  \textbf{26} (1974), 271--302.

\bibitem{Wa}
J.~Watanabe, 
\newblock A note on Gorenstein rings of embedding codimension three,
\newblock \emph{Nagoya Math. J.} \textbf{50} (1973), 227--232.
\end{thebibliography}

\end{document}